\documentclass[10pt,a4paper,draft]{article}

\usepackage{fullpage}
\usepackage{xparse} 
\usepackage{amsthm}
\usepackage{amsmath}
\usepackage{amssymb} 
\usepackage[all,cmtip]{xy} 
\usepackage[T1]{fontenc}    
\usepackage{textgreek} 
\usepackage{enumitem} 
\usepackage[inline,nomargin]{fixme} 
\fxsetup{targetlayout=color}
\usepackage{chngcntr} 
\usepackage{hyphenat} 
\usepackage{xspace}   
\usepackage{mathtools} 
\usepackage[final]{hyperref} 
\usepackage[capitalise,nameinlink]{cleveref} 

\def\Type{\hbox{\sf Type}}

\newcommand{\ints}{\mathbb{Z}}
\newcommand{\II}{\mathbb{I}}
\newcommand{\JJ}{\mathsf{J}}

\newcommand{\Unit}{\mathbf{1}}
\newcommand{\Nat}{\mathsf{Nat}}

\newcommand{\north}{\mathsf{north}}

\newcommand{\south}{\mathsf{south}}
\newcommand{\merid}{\mathsf{merid}}

\def\norm#1{\left\|#1\right\|}
\newcommand{\UU}{\mathcal{U}}

\newcommand{\Elem}{{\sf Elem}}
\newcommand{\Cov}{\mathsf{C}}
\newcommand{\CC}{{\mathcal C}}
\newcommand{\GG}{\mathsf{isMod}}

\newcommand{\patch}{\mathsf{patch}}
\newcommand{\leftinv}{\mathsf{linv}}

\newcommand{\msup}{\mathsf{sup}}

\newcommand{\id}{\mathsf{id}}

\newcommand{\const}{\mathsf{const}}

\newcommand{\hFiber}{\mathsf{fib}}

\newcommand{\isEquiv}{\mathsf{isEquiv}}
\newcommand{\BB}{\mathcal{B}}

\newcommand{\unit}{()} 

\newcommand{\zero}{\mathsf{zero}}
\newcommand{\SUCC}{\mathsf{succ}}

\newcommand{\rec}{\mathsf{rec}}

\newcommand{\unitF}{\langle \rangle}

\newcommand{\pairF}[2]{\langle #1,#2 \rangle}
\newcommand{\mpF}[1]{\langle #1 \rangle} 
\newcommand{\D}{D}              
\newcommand{\DD}{\tilde{D}} 
\newcommand{\dd}{\tilde{D}} 

\newcommand{\E}{E}
\newcommand{\EE}{\tilde{E}}

\newcommand{\PP}{\mathsf{P}}
\newcommand{\QQ}{\mathsf{Q}}
\newcommand{\RR}{\mathsf{R}}

\newcommand{\op}{\mathsf{op}}

\newcommand{\functorial}{\mathsf{ftl}}
\newcommand{\levelwise}{\mathsf{lw}}

\newcommand{\cofuniv}{\Phi}
\newcommand{\cofunivLW}{\Phi_{\levelwise}}
\newcommand{\cofunivConst}{\Phi}

\newcommand{\Gpd}{\mathsf{Gpd}}

\newcommand{\Modality}{\mathsf{Modality}}

\newcommand{\Set}{\mathsf{Set}}
\newcommand{\Psh}{\operatorname{PSh}}

\newcommand{\Cof}{\mathsf{Cof}}
\newcommand{\TrivCof}{\mathsf{TrivCof}}
\newcommand{\Fib}{\mathsf{Fib}}
\newcommand{\TrivFib}{\mathsf{TrivFib}}
\newcommand{\GenCof}{\mathsf{I}}
\newcommand{\GenTrivCof}{\mathsf{J}}
\newcommand{\WEquiv}{\mathsf{W}}
\newcommand{\InjFib}{\mathsf{InjFib}}
\newcommand{\InjTrivFib}{\mathsf{InjTrivFib}}

\newcommand{\Alg}{\mathsf{Alg}}
\newcommand{\Coalg}{\mathsf{Coalg}}
\newcommand{\Map}{\mathsf{Map}}

\newcommand{\CofFun}{\mathsf{C}}
\newcommand{\TrivCofFun}{\mathsf{C_t}}
\newcommand{\FibFun}{\mathsf{F}}
\newcommand{\TrivFibFun}{\mathsf{F_t}}

\newcommand{\obj}{\operatorname{obj}}
\newcommand{\aug}{\mathsf{aug}}
\newcommand{\Endo}{\mathsf{Endo}}
\newcommand{\gen}{\mathsf{gen}}

\newcommand{\M}{\mathcal{M}}
\newcommand{\join}{\mathop{\star}}

\newcommand{\Prop}{\mathsf{Prop}}

\newcommand{\mycomment}[1]{}

\makeatletter 
\newcommand{\varname}[1]{\begingroup\newmcodes@\mathit{#1}\endgroup} 
\newcommand{\DeclareAbbrevation}[2]{\newcommand{#1}{\@ifnextchar{.}{#2}{#2.\@\xspace}}}
\makeatother  

\newlist{conditions}{enumerate}{2}
\setlist[conditions,1]{label=\arabic*.,ref=\arabic*}
\setlist[conditions,2]{label=\arabic*.,ref=\arabic{conditionsi}.\arabic*}

\newlist{parts}{enumerate}{1}
\setlist[parts]{label=(\roman*)}

\newlist{options}{enumerate}{1}
\setlist[options]{label=\arabic*.,ref=\arabic*}

\newcounter{diagram}
\NewDocumentEnvironment{diagram}{oo}{
  \refstepcounter{diagram}
  \begin{center}
    \IfNoValueTF{#1}{
      \begin{tikzcd}
    }
    {
      \begin{tikzcd}[#1]
    }
}{
    \end{tikzcd}
    \IfNoValueF{#2}{
      \par\medskip
      Diagram~\thediagram: \textit{#2}
    }
  \end{center}
}
\NewDocumentEnvironment{diagram*}{o}{
  \begin{diagram}[#1]
}{
  \end{diagram}
}

\newcommand{\pullback}[1]{\save*!/#1-1.2pc/#1:(-1,1)@^{|-}\restore}

\DeclarePairedDelimiter\parens\lparen\rparen

\DeclarePairedDelimiter\verts\lvert\rvert

\DeclarePairedDelimiter\braces\lbrace\rbrace
\DeclarePairedDelimiterX\set[2]\lbrace\rbrace{#1 \mathrel{\delimsize\vert} #2}

\DeclareAbbrevation{\ie}{i.e}
\DeclareAbbrevation{\eg}{e.g}
\DeclareAbbrevation{\Eg}{E.g}
\DeclareAbbrevation{\cf}{cf}
\DeclareAbbrevation{\etc}{etc}
\DeclareAbbrevation{\resp}{resp}
\DeclareAbbrevation{\etal}{et al}
\DeclareAbbrevation{\ibid}{ibid}
\DeclareAbbrevation{\ca}{ca}
\DeclareAbbrevation{\vs}{vs}

\crefname{conditionsi}{condition}{condition}
\Crefname{conditionsi}{Condition}{Condition}
\crefname{conditionsii}{condition}{condition}
\Crefname{conditionsii}{Condition}{Condition}
\crefname{diagram}{diagram}{diagrams}
\Crefname{diagram}{Diagram}{Diagrams}
\crefname{figure}{figure}{figures}
\Crefname{figure}{Figure}{Figures}
\crefname{formula}{formula}{formulas}
\Crefname{formula}{Formula}{Formulas}
\crefname{partsi}{part}{parts}
\Crefname{partsi}{Part}{Parts}
\crefname{optionsi}{option}{options}
\Crefname{optionsi}{Option}{Options}
\crefname{section}{Section}{Sections}
\Crefname{section}{Section}{Sections}
\crefname{square}{Square}{Squares}
\Crefname{square}{Square}{Squares}
\crefname{subsection}{Subsection}{Subsections}
\Crefname{subsection}{Subsection}{Subsections}
\crefname{subsubsection}{Subsection}{Subsections}
\Crefname{subsubsection}{Subsection}{Subsections}

\newtheorem{theorem}{Theorem}[section]
\newtheorem{corollary}[theorem]{Corollary}
\newtheorem{lemma}[theorem]{Lemma}
\newtheorem{proposition}[theorem]{Proposition}

\theoremstyle{definition}
\newtheorem{definition}[theorem]{Definition}

\theoremstyle{remark} 
\newtheorem{remark}[theorem]{Remark}

\begin{document}

\title{Constructive Sheaf Models of Type Theory}

\author{Thierry Coquand, Fabian Ruch, and Christian Sattler}
\date{Computer Science Department, University of Gothenburg}
\maketitle


\section*{Introduction}

Despite being relatively recent, the notion of (pre\discretionary{-)}{}{)}sheaf model
has a rich and intricate history which mixes different intuitions
coming from topology, logic and algebra. Eilenberg and Zilber~\cite{EilenbergZ50}
used a presheaf model (simplicial sets) to represent
geometrical objects, and the intuition is {\em geometrical}:
we think of the objects $I,J,\dots$ of the base category as basic ``shapes'';
a presheaf $A$ is then given by a family of sets $A(I)$ of objects of each shape $I$,
which are related by the restriction maps $A(I) \rightarrow A(J)$.
A little later, but independently, Beth~\cite{Beth56} and Kripke~\cite{Kripke65}
used a sheaf and a presheaf model over trees, respectively,
to provide a formal semantics for intuitionistic logic. Their motivations
were logical, and the intuition is of a {\em temporal} nature instead: we think of the objects of the node of the tree
as ``stages of knowledge'' and of the ordering as ``increase in knowledge''.
Scott~\cite{Scott80} described a presheaf model of higher\hyp{}order logic and
pointed out the potential interest for the semantics of \textlambda\hyp{}calculus.
This was refined by Martin Hofmann~\cite{Hofmann97} who provided a presheaf
model of dependent type theory with universes. Hofmann's presheaf model was subsequently used in an
essential way in works on constructive semantics of type theory with
{\em univalent} universes~\cite{CCHM15,CHM18,OrtonP16}.

The generalization of such presheaf models of dependent type theory, and especially
of universes, to a {\em sheaf} model semantics
is however non\hyp{}trivial. The problem in generalising this semantics for universes
comes essentially from the fact that the collection of sheaves does not form a sheaf
in any natural way: if we are given locally sheaves that are compatible, one can patch
them together but not in a {\em unique} way, only unique {\em up to isomorphism}.
This problem was the motivation for the introduction of stacks and a more
subtle notion of patching of sheaves (\cf~\cite[Section~3.3]{EGAI}), and in general
patching of mathematical structures. The generalization of this to patching of
higher structures was the content of the first part of Joyal's letter to Grothendieck~\cite{Joyal84}.
One contribution of the present paper is to provide
a constructive version of this notion\footnote{Joyal's argument was using non\hyp{}constructive reasoning in simplicial sets
  and then Barr's theorem (see~\cite{Barr74}).
  The present paper can be developed directly in the constructive framework of CZF with universes
introduced by Aczel~\cite{Aczel98}.}
by describing a
sheaf model semantics of type theory with univalence~\cite{Voevodsky15,HoTT}.
This uses in a crucial way the fact that
we have a {\em constructive} interpretation of univalence as in~\cite{CCHM15,OrtonP16}, which can
be relativized to any presheaf model. The main point is then that the operation sending an object to its object of descent data (a compatible collection of elements of its restrictions) defines
a {\em left exact modality} (see~\cite{HoTT,Quirin16,RijkeSS17}), which can then be used
to build internally models of univalent type theory~\cite{Quirin16}.

This work opens the possibility of generalising works of sheaf models of intuitionistic
logic as in~\cite{TroelstraD88b} to sheaf models of univalent type theory. It extends the previous
work in~\cite{CoquandMR17} to a complete model of univalence, and has no restrictions for
representing (higher) data types.
We give only one application (independence of countable choice), but we expect for instance that results
such as in~\cite{MannaaC13} can be generalized as well, and that we can give a constructive
account of works such as in~\cite{Shulman18,Wellen17}.
The present semantics (in a preliminary version)
has already been used by Weaver and Licata~\cite{WeaverL20} for building a constructive model of directed univalence.

This paper is organized as follow. We first introduce the notion of
lex operation as an operation acting on types and families of types.
A descent data operation is then
a lex operation which defines a left exact modality \cite{HoTT,Quirin16,RijkeSS17}.
These two notions are formulated purely syntactically in the framework of type theory.
We show next how to instantiate these operations for cubical presheaves. In this setting,
we can understand the
notion of being modal for a descent data operation as a generalization of the sheaf condition,
where the compatibility requirements are expressed up to path equality instead of
being expressed as strict equalities. We then provide some examples and the application to
the unprovability of countable choice. In an appendix we explain how some
of our results about descent data operations can be generalized to accessible
left exact modalities.


\section{Abstract notion of descent data}

In this article, we take terminology in type theory with potentially both strict and homotopical meaning to have the strict meaning by default.
For example, equality (denoted by the symbol ${=}$) refers to the strict equality (as opposed to identity or path types), isomorphisms refer to strictly invertible maps, and pullbacks refer to strict pullbacks.

We use the following notations.
We write $\Unit$ for the unit type and $\unit : \Unit$ for its unique element.
Given a type $A$ and a family $B$ of types over $A$, we write $\sum_A B$ for their sum type and $\prod_A B$ for their product type.
The pairing operation is denoted by $(a, b) : \sum_A B$ for $a : A$ and $b : B\,a$.
The projection maps are denoted by $\pi_1$ and $\pi_2$.
We write $\id_A$ for the identity function on $A$ and $g\circ f$ for the composition of $f \colon A \to B$ and $g \colon B \to C$.
If $B$ is a family of types over $A$ and $f\colon A'\rightarrow A$, we also write $B\circ f$ for the family of types over $A'$ obtained from $B$ by reindexing along $f$.

\subsection{Lex operation} \label{subsec:lex-operation}

The concept of lex operations is defined for a dependent type theory with only unit type, dependent sums, dependent product and universes.
In particular, path types are not needed.
Intuitively, a lex operation is an endofunctor on the category of types and functions (compatible with substitution) which preserves the unit type and dependent projections of sum types up to isomorphism.

A {\em lex operation}\footnote{The notion of lex operation appears implicitly in a natural way
when describing the rules of inductive data types~\cite{CoquandP88}. If we have a family
$\D_a$ of lex operations indexed over $a:A$, we can consider the inductive type $T$
with constructor $\msup:\prod_{a:A} (\D_aT\rightarrow T)$ and elimination rule
$\rec\,f:\prod_TP$ for $f:\prod_{a:A} \prod_{u:\D_aT} (\DD_aP\,u\rightarrow P(\msup_a\,u))$.
We can then write the computation rule
\begin{equation*}
\rec\,f\,(\msup_a\,u) = f_a\,u\,(\dd_a(\rec\,f)\,u)
\end{equation*}
For justifying the use of such inductive definitions, we need some ``accessibility'' assumption
on the functor $\D$, which will be satisfied in the examples.
In the special case where $\D_aX$ is $X^{Ba}$ for $B$ a family of over $A$ we recover
the $W$\hyp{}type $W_A B$.}
is given by an operation $\D$ on types and functions forming a functor:
we have $\D f:\D A\rightarrow \D B$ if $f:A\rightarrow B$ with $\D(g\circ f) = \D g\circ \D f$
and $\D(\id_A) = \id_{\D A}$.

The operation $\D$ should also preserve the unit type $\Unit$ up to isomorphism.
Specifically we have an element $\unitF$ in $\D\Unit$ and $x = \unitF$ if $x$ is in $\D\Unit$.

Furthermore, $\D$ should preserve dependent projections of sum types up to isomorphism.
We assert this by an operation on families of types: $\DD B$ is family of types over $\D A$
if $B$ is a family of types over $A$.
This should be natural in $A$
together with operations ensuring that $\D(\sum_AB)$ is naturally
isomorphic to $\sum_{\D A}{\DD B}$ over $\D A$.
Naturality in $A$ means $\DD (B\circ f) = \DD B\circ \D f:\D{A'}\rightarrow\UU$ for $f:A'\rightarrow A$.
The natural isomorphism between $\D(\sum_AB)$ and $\sum_{\D A}{\DD B}$ over $\D A$ is given by
an operation $\dd s : \prod_{\D A}{\DD B}$ on sections $s:\prod_A B$
satisfying $\dd (s\circ f) = \dd s \circ \D f : \prod_{\D{A'}} \DD (B\circ f)$
and a pairing operation $\pairF{u}{v} : \D(\sum_A B)$ for elements $u:\D A$ and $v:(\DD B)\,u$
satisfying
\begin{align*}
(\D\pi_1)\pairF{u}{v} &= u &
(\dd\pi_2)\pairF{u}{v} &= v &
\pairF {(\D\pi_1)\,w}{(\dd\pi_2)\,w} &= w
\end{align*}
where $w:\D(\sum_AB)$.

We also assume that universes reflect these operations.
This means we have $\D A:\UU$ if $A:\UU$ and $\DD B:\D A\rightarrow\UU$ if $B:A\rightarrow \UU$ and $A$ a type (crucially, $A$ need not be in $\UU$ here).

\medskip

The canonical example of a lex operation is exponentiation with a fixed type $R$
(assumed to be in all universes).
We define $\D A = A^R$, $(\DD B)u = \prod_{x:R}\,B(u\,x)$, and $(\dd s)u = \lambda_{x:R}\,s(u\,x)$.
The pairing is given by $\pairF{u}{v} = \lambda_{x:R}\,(u\,x,v\,x)$.

\begin{remark}
  Let $\UU$ be a universe.
  The action of the operation $\DD$ on $\UU$\hyp{}small families is uniquely determined by the universal case $L = \DD\,\id_\UU : \D \UU \to \UU$: we have (and can define) $\DD B = L\circ \D B$ with $\D B:\D A \rightarrow \D\UU$ for $B:A\rightarrow\UU$.
  This corresponds to the ``escaping'' function in Section 2.5 of \cite{SchreiberS14}.
  We can thus describe the action of $\DD$ on $\UU$\hyp{}small families and associated operations by requiring that $\D$ applied to the ``universal $\UU$\hyp{}small fibration'' $\sum_{X:\UU} X \to \UU$ is isomorphic to a ``$\UU$\hyp{}small fibration'' (a projection of a type in $\UU$), and that $\D$ preserves pullbacks of this map.
\end{remark}

\begin{proposition} \label{lex-operation-pointed}
  Any lex operation $\D$ is uniquely pointed.%
  \footnote{We owe this observation to Dan Licata.}
\end{proposition}

\begin{proof}
  We define $\eta_A\, a = (\D\epsilon_a)\,\unitF$ with $\epsilon_a = \lambda_{x:\Unit}\,a$.
  We then have for $f:A\rightarrow B$
  \begin{equation*}
  (\D f)(\eta_A\, a) = (\D f\circ \D\epsilon_a)\,\unitF = (\D\epsilon_{f\,a})\,\unitF = \eta_B(f\,a)
  \end{equation*}
  Note furthermore that this natural transformation $\eta_A$ is uniquely determined,
  since we should have $\eta_{\Unit}\,\unit = \unitF$ and so
  $\eta_A\, a = \eta_A(\epsilon_a\,\unit) = (\D\epsilon_a)(\eta_{\Unit}\,\unit) = (\D\epsilon_a)\,\unitF$.
\end{proof}

\begin{remark}\label{iso}
  For $T : \UU$, the map
  $(\DD \epsilon_T)\,\unitF\rightarrow \sum_{\D\Unit}\DD \epsilon_T\rightarrow
  \D(\sum_{\Unit}\epsilon_T)\rightarrow \D T$ is an isomorphism, as a composition
  of isomorphisms.
  As a consequence, the map
  \begin{equation*}
  (\DD B)(\eta_A\, a)\longrightarrow \D(B\,a)\qquad v\longmapsto (\D\pi_2)\pairF{\unitF}{v}
  \end{equation*}
  is an isomorphism for a type $A$, a family $B$ over $A$, and $a : A$.
\end{remark}





\medskip

Note that $\eta_A\,a = \lambda_{x:R}\,a$ for the example $\D A = A^R$
where the lex operation is exponentiation.
The isomorphisms of \cref{iso} are identities in this example.%
\footnote{This assumes that function types are implemented via dependent products.}
In fact, this will happen for all the example lex operations we will consider in this article.

\begin{remark} \label{D-reflection}
Recall our assumption that universes reflect the operation $\D$ on types and the operation $\DD$ on families.
\Cref{iso} shows that only the reflection of $\DD$ is essential.
If $\D$ is not reflected, we can define an isomorphic operation $\D' A = \DD (\epsilon_A)\, \unitF$ on types that is reflected.
The remaining structure of $\D$ transports across the isomorphism to define a lex operation $\D'$.
\end{remark}

\begin{remark}
Let $\mathcal{E}$ be a category with families~\cite{Dybjer95} modelling our type theory.
A lex operation in $\mathcal{E}$ can be defined from a pseudomorphism of cwfs with universes~\cite{KaposiHS19} from $\mathcal{E}$ to itself that is pointed as an endofunctor.
When working externally with a model, this is a convenient way of constructing a lex operation in it.
Note that the given pointing is then reconstructed by \cref{lex-operation-pointed}.
\end{remark}

\begin{remark}
For readers familiar with Martin Hofmann's semantic methodology~\cite{Hofmann97}, we note a concise description of lex operations expressed internally in presheaves over the category of contexts.
The object $\Type$ of types has the structure of a cwf with universes (context extension is given by sum types).
Up to the discussion of \cref{D-reflection}, a lex operation is a pseudomorphism of cwfs with universes from $\Type$ to itself.
This definition can be written in the language of two\hyp{}level type theory~\cite{AnnenkovCKS17}.
\end{remark}

\subsection{\texorpdfstring{$\D$}{D}\hyp{}modal types}

The notion of lex operation is defined at the level of ``pure'' dependent type theory,
without assuming any notion of path types. In presence of path types, we automatically
have the following preservation property.

\begin{theorem}\label{lex}
  Let $\D$ be a lex operation.
  Then $\D$ preserves equivalences.
\end{theorem}

\begin{proof}
  Note that if $f_0$ and $f_1$ are path equal then so are $\D f_0$ and $\D f_1$ by path induction.
  It follows that if $f$ and $g$ are inverses, then so are $\D f$ and $\D g$.
\end{proof}

Avigad \etal~\cite{AvigadKL15} explain how to build a fibration category from a model of
dependent type theory. \Cref{lex} implies that any lex operation defines an
endomorphism of the associated fibration category. A lex operation preserves
all finite homotopy limits (\eg, contractible types, homotopy pullbacks, homotopy equalizers, homotopy fibers, \ldots).

\medskip

In presence of path types, we can also define the following important notion of modal types.

\begin{definition}
  A type $A$ is called $\D$\hyp{}modal if the unit map $\eta_A:A\rightarrow \D A$ is an equivalence.
\end{definition}

\begin{proposition}\label{sum}
  If $A$ is $\D$\hyp{}modal and $B$ is a family of types over $A$, then
  $B$ is a family of $\D$\hyp{}modal types over $A$ if, and only if, $T = \sum_AB$ is $\D$\hyp{}modal.
\end{proposition}

\begin{proof}
  Let $f$ be the map
  $T\rightarrow \sum_{\D A}{\DD B},\ (a,b)\mapsto (\eta_A\,a,\eta_{B\,a}\,b)$.
  Since $\eta_A$ is an equivalence,
  each map $\eta_{B\,a}$ is an equivalence if, and only if, the map $f$ is an equivalence~\cite{HoTT}.
  But $f$ is an equivalence if and only if $\eta_{T}$ is an equivalence.
\end{proof}

\subsection{Abstract notion of descent data}


\begin{theorem}\label{main}
  The following conditions are equivalent, for a lex operation $\D$
  \begin{enumerate}
  \item $\D$ defines a modality as axiomatized in~\cite{Quirin16,RijkeSS17}
  \item the map $\D\eta_A$ is an equivalence, and $\D\eta_A$ and $\eta_{\D A}$
    are path equal
  \end{enumerate}
\end{theorem}

\begin{proof}
  The first condition implies the second using the results in~\cite{Quirin16,RijkeSS17}.

  Conversely, assume that the map $\D\eta_A$ is an equivalence, and $\D\eta_A$ and $\eta_{\D A}$
  are path equal. Then $\eta_{\D A}$ is an equivalence as well and each type $\D A$ is $\D$\hyp{}modal.
  \Cref{sum} shows that $\D$\hyp{}modal types are closed by dependent sum types.
  We thus only have to prove that the map
  \begin{equation*}
  F:(\D A\rightarrow B)\longrightarrow (A\rightarrow B) \qquad f\longmapsto f\circ\eta_A
  \end{equation*}
  is an equivalence if $B$ is $\D$\hyp{}modal~\cite{Quirin16}.

  Let $p_B$ be a map $\D B\rightarrow B$ such that $p_B\circ\eta_B$
  is path equal to $\id_B$. We define a map
  \begin{equation*}
  G:(A\rightarrow B)\longrightarrow (\D A\rightarrow B)\qquad u\longmapsto p_B\circ \D u
  \end{equation*}
  We then have $F(Gu) = p_B\circ \D u\circ\eta_A = p_B\circ\eta_B\circ u$ which is path equal to $u$
  and $G(Ff) = p_B\circ \D(f\circ \eta_A) = p_B\circ \D f\circ \D\eta_A$ which is path equal
  to $p_B\circ \D f\circ \eta_{\D A} = p_B\circ\eta_B\circ f$ which is path equal to $f$. Hence $G$
  is an inverse to $F$ and $F$ is an equivalence.
\end{proof}

\begin{definition}
  A {\em descent data operation} is a lex operation $\D$
  satisfying the equivalent conditions of \cref{main}.
\end{definition}

Note that the first condition of \cref{main} is a (homotopy)
proposition. The second condition is the one which will be convenient to verify
for the main examples.


We write $\GG_\D(A)$ for the type (proposition) expressing that $A$ is $\D$\hyp{}modal.

\subsection{Closure properties}

 Let $\D$ be a descent data operation. 

\begin{lemma}\label{suff}
  For the map $\eta_A:A\rightarrow \D A$ to be an equivalence, it is enough to have
  a patch function $p_A:\D A\rightarrow A$ such that $p_A\circ\eta_A$ is path equal to the identity of $A$.
\end{lemma}

\begin{proof}
  If $p_A$ is such a patch function, we have
  $\id_{\D A} = \D(\id_A) = \D(p_A\circ\eta_A) = \D p_A\circ \D\eta_A$ which is path equal
  to $\D p_A\circ\eta_{\D A} = \eta_A\circ p_A$. Hence $p_A$ is an inverse of $\eta_A$ and $\eta_A$
  is an equivalence.
\end{proof}

\begin{lemma}\label{DD}
For $B$ a family of types over $A$ and any $u:\D A$, the type $(\DD B)\,u$ is $\D$\hyp{}modal.
\end{lemma}

\begin{proof}
  Since $\D(\sum_AB)$ is $\D$\hyp{}modal, so is the isomorphic type
$\sum_{\D A}{\DD B}$.
Using \cref{sum}, we have that $(\DD B)\,u$
is a $\D$\hyp{}modal type for any $u:\D A$.
\end{proof}

\begin{proposition}\label{univ}
The type $\UU_\D = \sum_{\UU}\GG_\D$ is a $\D$\hyp{}modal type.
\end{proposition}

\begin{proof}
  Consider the diagram
\[
\xymatrix{
  \D\UU_\D \ar[dr]^{\DD\pi_1} \\
  \UU_\D \ar[u]^{\eta} \ar[r]_{\pi_1} & \UU \rlap{.}
}
\]
It commutes up to homotopy since $(\DD\pi_1)(\eta\, X)$ is isomorphic to
$\D(\pi_1\, X)$ by \cref{iso}, which
is path equal to $\pi_1\, X$ for any $X:\UU_\D$ by univalence(!). Note also that $\pi_1$
is an embedding since $\GG_\D$ is a family of propositions.

Since $(\DD\pi_1)\, A$ is $\D$\hyp{}modal by \cref{DD} for any $A:\D\UU_\D$,
the map $\DD\pi_1:\D\UU_\D\rightarrow\UU$
factorizes through $\pi_1:\UU_\D\rightarrow\UU$ and the corresponding
map $\D\UU_\D\rightarrow\UU_\D$ is a left inverse of $\eta:\UU_\D\rightarrow \D\UU_\D$
since $\pi_1$ is an embedding. Hence $\UU_\D$ is $\D$\hyp{}modal by \cref{suff}.
\end{proof}

\subsection{Model associated to a descent data operation}

We can now define an internal translation which provides a new model of
univalent type theory with higher inductive types
for any descent data operation $\D$, following the work in~\cite{Quirin16}. A
type $A,p$ of the new model
is a type $A$ {\em together with}
a proof $p$ that this type is $\D$\hyp{}modal, while an element of a pair $A,p$ is an element of $A$.






In order to interpret the type of natural numbers with the desired
computation rules (not covered in~\cite{Quirin16}),
we need to use the following higher inductive type:%
\footnote{To justify the use
  of such inductive definitions, we need some accessibility assumption on the functor
  $\D$ that will be satisfied in the examples.}
\begin{alignat*}{3}
  & \zero     & \quad & : & \quad & \Nat \\
  & \SUCC     & \quad & : & \quad & \Nat \rightarrow \Nat \\
  & \patch    & \quad & : & \quad & \D\,\Nat \rightarrow \Nat \\
  & \leftinv  & \quad & : & \quad & \textstyle\prod_{x : \Nat} \patch (\eta_\Nat\, x) =_{\Nat} x
\end{alignat*}

This is equivalent to the type $\D N$ where $N$ is the usual inductive
type with constructors $\zero$ and $\SUCC$, but the type $\D N$ does not satisfy
the required computation rules.

The same idea
applies to the interpretation of other inductive types
such as the $W$\hyp{}type.

\medskip

It also works for higher inductive types. For instance
the suspension of a type $A$ will be defined as
\begin{alignat*}{3}
  & \north, \south    & \quad & : & \quad & T \\
  & \merid     & \quad & : & \quad & {A}\rightarrow \north =_{T} \south \\
  & \patch    & \quad & : & \quad & \D T \rightarrow T \\
  & \leftinv  & \quad & : & \quad & \textstyle\prod_{z : T} \patch (\eta_{T}\, z) =_{T} z
\end{alignat*}

Note that having $\D$ defined as a {\em strict} functor is essential for such definitions.

\subsection{Generalization to a family of descent data operations}\label{Suniv}

More generally, if we have a {\em family} of descent data operations $\D_S$ indexed
by a given type $S:\Cov$, with corresponding maps $\eta^S_A:A\rightarrow \D_SA$,
we can consider $\GG_{\Cov}(A)$ to be the proposition $\prod_{S:\Cov} \GG_{\D_S}(A)$
and $\UU(\Cov)$ which is $\sum_{\UU}\GG_{\Cov}$. We use the slightly shorter notation $\UU_S$ to denote the previously defined type $\UU_{\D_S} = \sum_{\UU}\GG_{\D_S}$

We let the preorder $\D_1\leqslant \D_2$ on descent data operations mean that any $\D_1$\hyp{}modal
type is $\D_2$\hyp{}modal. We say that $\Cov$ is {\em filtered} if we have
$\exists_{S:\Cov}.\,\D_S\leqslant \D_{S_1}\wedge \D_S\leqslant \D_{S_2}$ for any
$S_1, S_2:\Cov$.%
\footnote{Existence is defined as the propositional truncation of the dependent sum type~\cite{HoTT}.}

\begin{theorem}\label{filter}
If $\Cov$ is filtered then $\UU(\Cov)$ satisfies $\GG_{\Cov}$.
\end{theorem}

\begin{proof}
  For any $\D_S\leqslant \D_{S_1}$ in $\Cov$, $\UU_S$ is $\D_S$\hyp{}modal by \cref{univ}
  and so $\D_{S_1}$\hyp{}modal,
  and hence $\eta^{S_1}:\UU_{S}\rightarrow \D_{S_1}\UU_{S}$ has an inverse.
  It follows that the map $\DD_{S_1}\pi_1:\D_{S_1}\UU_{\Cov}\rightarrow \UU$
  factorizes through $\UU_S\rightarrow\UU$ and
  hence that for any $A:\D_{S_1}\UU_{\Cov}$ the type $(\DD_{S_1}\pi_1)\,A$
  is $\D_S$\hyp{}modal.

  If $\Cov$ is filtered, this implies that the type
  $(\DD_{S_1}\pi_1)\,A$ is $\D_{S_2}$\hyp{}modal for any $S_2$ in $\Cov$. Hence
  the map $\DD_{S_1}\pi_1:\D_{S_1}\UU_{\Cov}\rightarrow \UU$
  factorizes through $\UU_{\Cov}\rightarrow \UU$
  and the corresponding map $\DD_{S_1}\UU_{\Cov}\rightarrow\UU_{\Cov}$
  is a left inverse of $\UU_{\Cov}\rightarrow \D_{S_1}\UU_{\Cov}$.
  Hence $\UU(\Cov)$ is $\D_{S_1}$\hyp{}modal for any $S_1$ in $\Cov$
  by \cref{suff}.
\end{proof}

This shows that for a family of descent data operations satisfying the hypothesis of \cref{filter},
we still get a model of univalent type theory (with higher inductive types), interpreting a type
as a type together with a proof that this type is modal for each descent data operation.
Of all type formers, only the universe has to deal with interaction between the elements of the given family of descent data operations.

\subsection{Example}\label{exp}

If $R$ is a {\em proposition}, then for the lex operation defined by $\D A = A^R$ the two maps
$\D\eta_A$ and $\eta_{\D A}$ are path equal equivalences and hence exponentiation
defines a descent data operation in that case.

 The next \lcnamecref{sec:presheaf-models} will define a new kind of descent data operation for any presheaf model.

\section{Cubical presheaf models}\label{sec:presheaf-models}

\subsection{Cubical models}\label{cubical-sets}

Cubical models are presheaf models of univalent type theory specified by two parameters, an \emph{interval object} $\II$ and a \emph{cofibration classifier} $\cofuniv$.
Formally, we say that a \emph{cubical model} is a presheaf category with the following structure, as in Orton and Pitts~\cite{OrtonP16}.
\begin{itemize}
\item
The interval object $\II$ is connected and has distinct points $0$ and $1$.
Exponentiation with $\II$ has a right adjoint.
We also assume that $\II$ has the structure of a bounded distributive lattice.%
\footnote{%
This assumption simplifies one of our arguments (\cref{key2}).
However, our results also apply to the Cartesian variation of cubical models of Angiuli \etal~\cite{AngiuliBCHHL}).
There, one removes this hypothesis and instead adds that the diagonal $\II \to \II \times \II$ is a cofibration.
}
\item
The universal cofibration $\top \colon 1 \to \cofuniv$ is a levelwise decidable inclusion.
In the internal language of presheaves, we will work with $\cofuniv$ as a universe of certain propositions and leave the decoding function (given by equality with $\top : \cofuniv$) implicit.
Isomorphic cofibrations are equal.%
\footnote{This assumption is not strictly speaking necessary, but simplifies the theory.}
The interval endpoint inclusions $0, 1 \colon 1 \to \II$ are cofibrations.
Cofibrations are closed under finite union (finite disjunction), composition (dependent conjunction), and universal quantification over $\II$.
\end{itemize}
It is then known, following the work in~\cite{CCHM15,OrtonP16,CHM18}, how to define a model of univalent type theory with higher inductive types.

\subsection{Presheaves in cubical models}\label{subsec:presheaves-in-cubical-sets}

For the remainder of this \lcnamecref{sec:presheaf-models}, we fix a cubical model given by presheaves over a small category $\BB$.
We refer to this as the \emph{base model} (for example, it can be cubical sets).
We write $I,J,K,\dots$ for the objects of $\BB$.
We have the interval object by $\II_{\BB}$ and the cofibration classifier $\cofuniv_{\BB}$.


Let $\CC$ be another small category.
We write $X,Y,Z,\dots$ for its objects.
We describe possibilities for turning presheaves over $\CC \times \BB$ into a cubical model.
For the interval object $\II$, we simply take $\II(X,I) = \II_{\BB}(I)$.
For the cofibration classifier, we have two reasonable options:
\begin{options}
\item
The first example is simply to take $\cofunivConst(X,I) = \cofuniv_{\BB}(I)$.
\item
The second example is to define an element $\psi$ of $\cofunivLW(X,I)$ to be
a family $\psi_f$ in $\cofuniv_{\BB}(I)$ for $f \colon Y\rightarrow X$ such that
$\psi_f\leqslant \psi_g$ if furthermore $g \colon Z\rightarrow Y$.
We then define the restriction operation $\psi (f,l)$ to be the family
$\psi(f,l)_g = \psi_{f g} l$ for $f \colon Y\rightarrow X$ and $g \colon Z\rightarrow Y$.
\end{options}
The motivation for the second example is that if
$\cofuniv_{\BB}(I)$ is the collection of (decidable) sieves on $I$, then
$\cofunivLW(X,I)$ becomes the collection of (decidable) sieves on $(X,I)$.%
\footnote{Classically, this corresponds to having all monomorphisms as cofibrations.}

The interval object $\II$ and any of the choices $\cofunivConst$ and $\cofunivLW$ fit all the requirements listed in \cref{cubical-sets}.
This turns presheaves over $\CC \times \BB$ into a cubical model.
In particular, we get a model of univalent type theory (and higher inductive types).
We are going to analyse the model obtained using the choice $\cofunivConst$ for the cofibration classifier and then indicate how to adapt these results for $\cofunivLW$.

\medskip

In this model, a context $\Gamma$ is interpreted by a presheaf
over $\CC\times\BB$ so a family of sets $\Gamma(X,I)$ with suitable
restriction maps $\rho\mapsto \rho(f,l)$ with $f:Y\rightarrow X$ in $\CC$ and
$l:J\rightarrow I$ in $\BB$.

A dependent type $A$ over $\Gamma$ is then given by a presheaf over the category
of elements of $\Gamma$: for any $\rho$ in $\Gamma(X,I)$ we have a set
$A\rho$ with suitable restriction maps $A\rho \rightarrow A\rho(f,l)$ denoted by $u\mapsto u(f,l)$
together with a filling operation (see~\cite{CCHM15,OrtonP16}).
We write $\Type(\Gamma)$ for the collection of all types with a composition operation
over $\Gamma$.
The set $\Elem(\Gamma,A)$ is then the set of sections: a family $a\rho$
in $A\rho$ such that $(a\rho)(f,l) = a(\rho(f,l))$ for any $\rho$
in $\Gamma(X,I)$ and $f,l$ map of codomain $X,I$.


Given a constructive Grothendieck universe $U$ (see~\cite{Aczel98}) containing $\BB$ and $\CC$,
we write $\Type_{U}(\Gamma)$ for the set of $U$\hyp{}types, such that each set
$A\rho$ is in $U$. The presheaf $\Type_U$ is then represented by a fibrant
type $\UU$, which is univalent~\cite{CCHM15}.

\subsection{Internal language description}\label{subsec:internal-language-description}

This was an external description of the presheaf model.
It is also possible to describe this model
using
the internal logic of the presheaf topos over $\CC\times\BB$ as in~\cite{OrtonP16,CHM18}
but also using
the internal logic of the presheaf topos over $\BB$.
We will use both descriptions.

In the internal logic of the presheaf topos over $\BB$,
a context of the presheaf model over $\CC$ is interpreted as a family of ``spaces'' $\Gamma(X)$
with restriction maps $\rho\mapsto\rho f$ for $f:Y\rightarrow X$.
(Each space $\Gamma(X)$ is itself a presheaf over $\BB$ with $\Gamma(X)(I) = \Gamma(X,I)$.)
A dependent type $A$ over $\Gamma$ is given by a family of spaces $A\rho$
for $\rho$ in $\Gamma(X)$ with restriction maps $u\mapsto uf$. The presheaf
$\cofunivConst$ of cofibration is the constant presheaf $\cofunivConst(X) = \cofuniv_{\BB}$.
The interval $\II$ is the constant interval $\II(X) = \II_{\BB}$.

It will be convenient to introduce the following notation: if
$\gamma$ is an element of $\Gamma(X)^{\II}$ and
$f:Y\rightarrow X$, we write $\gamma f^+$ in $\Gamma(Y)^{\II}$
for $\lambda_i\,\gamma(i)f$. Similarly if $u(i)$ is a section
in $A\gamma(i)$ we write $uf^+$ for $\lambda_i\,u(i)f$.


A {\em filling operation}
(see~\cite{OrtonP16,CHM18}) for $A$ is given by an operation $c_A$ which takes as
argument $\gamma$ in $\Gamma(X)^{\II_{\BB}}$ and $\psi$ in $\cofunivConst(X) = \cofuniv_{\BB}$
and a family of elements $u(i)$ in $A\gamma(i)f$ on the extent
$\psi\vee i = 0$. (There is a dual operation with $i=1$ instead.)
It produces an element $c_A(X,\gamma,\psi,u)(i)$ in $A\gamma(i)$ such that
\begin{conditions}
\item $c_A(X,\gamma,\psi,u)(i) = u(i)$ on $\psi\vee i=0$,
\item $c_A(X,\gamma,\psi,u)(i)f = c_A(Y,\gamma f^+,\psi, u f^+)(i)$ for $f:Y\rightarrow X$.
\end{conditions}
Given such an operation, we also call $A$ \emph{fibrant} (note this is structure rather than property).

If $A$ is a type over $\Gamma$, we get a family of dependent types $A(X)$
over $\Gamma(X)$, each of them having a filling operation, but furthermore
these filling operations commute with the restriction maps.

Similarly an {\em extension operation} for $A$, witnessing that $A$ is contractible
(see~\cite{CCHM15}), is given by an operation $e_A$ which takes as argument $\rho$
in $\Gamma(X)$ and a partial element $u$ on the extent $\psi$ and produces
an element $e_A(X,\rho,\psi,u)$ in $A\rho$ such that
\begin{conditions}
\item $e_A(X,\rho,\psi,u) = u$ on $\psi$,
\item $e_A(X,\rho,\psi,u)f = e_A(Y,\rho f,\psi, uf)$ for $f:Y\rightarrow X$.
\end{conditions}
Given such an operation, we also call $A$ \emph{trivially fibrant} (again, this is structure rather than property).

If $A$ is contractible, each $A(X)$ is a contractible family of types over $\Gamma(X)$.
But conversely, it may be that each $A(X)$ has an extension operation $e_A(X)$ which
does not commute with restriction (see \cref{subsec:examples}).
Similarly, a map $\sigma:A\rightarrow B$
which is an equivalence defines a family of equivalences $\sigma_X:A(X)\rightarrow B(X)$
but it may be that each map $\sigma_X$ is an equivalence, without $\sigma$ being
an equivalence.

\begin{remark} \label{presheaf-model-comparison}
We have a canonical map from $\cofunivConst$ to $\cofunivLW$ sending $\psi : \cofunivConst(X)$ to the constant family on $\psi$.
This map commutes (up to isomorphism) with the decoding to propositions.
It follows that there is a natural map from extension operations for $\cofunivConst$ to extension operations for $\cofunivLW$, and the same holds for filling operations.
It follows that a (contractible) type for the cubical presheaf model for $\cofunivConst$ is naturally also a (contractible) type for the cubical presheaf model for $\cofunivLW$.
\end{remark}

\begin{remark} \label{presheaf-model-groupoid}
Let $\CC$ be a groupoid.
Then for $\psi : \cofunivLW(X)$ and $f : Y \to X$, we have $\psi_{\id_X} \leq \psi_f \leq \psi_{f f^{-1}} = \psi_{\id_X}.$
It follows that $\psi$ is the constant family on $\psi_{\id_X}$.
Thus, the map $\cofunivConst \to \cofunivLW$ from \cref{presheaf-model-comparison} is invertible.
It follows that the cubical presheaf models for $\cofunivConst$ and $\cofunivLW$ are the same.
We thank Emily Riehl for this observation.
\end{remark}

\subsection{Examples} \label{subsec:examples}

Let $\BB$ be a concrete cube category, for instance the Cartesian~\cite{AngiuliBCHHL}, distributive lattice, or de Morgan one~\cite{OrtonP16,CCHM15}.
Then we have a nerve functor from groupoids to cubical sets in the sense of presheaves over $\BB$.
In this way, we can see any groupoid as a cubical set with a canonical filling operation.

For the first example, let $\CC$ be the group $\ints/(2)$.
Let $\tau$ be the non\hyp{}trivial element of this group.
A context is a space with an involutive action $\rho\mapsto\rho\tau$.
A dependent type $A$ over $\Gamma$ has also an involutive action $A\rho\rightarrow A\rho\tau$ denoted by $u\mapsto u\tau$ with a filling operation which is equivariant, meaning $c_A(\gamma,\psi,u)(i)\tau = c_A(\gamma \tau^+,\psi, u\tau^+)(i)$.
Let $A$ be the groupoid with two isomorphic objects swapped by $\tau$.
Then $A$ is pointwise contractible, but is not contractible in the presheaf model, since it has no global point.
Another way to describe this example is that the unique map $A\rightarrow \Unit$ is a pointwise equivalence, but is not an equivalence.

For the second example, let $\CC$ be the poset on objects $\bot, 0, 1$ with $\bot < 0$ and $\bot < 1$.
We define a global type $A$ as follows.
We take $A(0)$ and $A(1)$ to consist of a single object $a_0$ and $a_1$, respectively.
We take $A(\bot)$ to consist of an isomorphism between the restrictions of $a_0$ and $a_1$.
Then $A$ is levelwise contractible (\ie, $A(\bot), A(0), A(1)$ are contractible), but $A$ is not contractible since it has no global point.

We note that the second example is fixed by working with the cofibration classifier $\cofunivLW$.
However, as explained by \cref{presheaf-model-groupoid}, this does not apply to the first example.

\section{Homotopy descent data}

\subsection{A lex operation}
\label{sec:a-lex-operation}

In this \lcnamecref{sec:a-lex-operation}, we work in the internal language of the presheaf topos
over $\BB$.
We first define a lex operation on presheaf types, and then show that
this lex operation extends to types with a filling operation.

For any $A$ presheaf over $\Gamma$ we define $\E A$ presheaf over $\Gamma$.
An element $u$ of $(\E A)\rho$, for $\rho$ in $\Gamma(X)$ is given
by a family of elements $u(f)$ in $A\rho f$ for $f:Y\rightarrow X$.
We define the restriction $uf$ in $(\E A)\rho f$ by $uf(g) = u(fg)$
if $f:Y\rightarrow X$ and $g:Z\rightarrow Y$.

If $B$ is presheaf over $\Gamma.A$, we define $\EE(B)$ presheaf over  $\Gamma.{\E A}$. If
$\rho$ is in $\Gamma(X)$ and $u$ is in $(\E A)\rho$, then $\EE(B)(\rho,u)$ is the space of families
$v(f)$ in $B(\rho f,u(f))$.

We define a natural transformation $\alpha:A\rightarrow \E A$ by $(\alpha a)(f) = af$.

Next, we extend the action of $\E$ to types with a filling operation. Actually,
we define a filling operation $\E(c_A)$ on $\E A$ assuming only that $c_A$
is a {\em pointwise} filling operation on $A$.

\begin{proposition}\label{Efill}
  We can define a filling operation $\E(c_A)$ on $\E A$ if $c_A$ is a pointwise
  filling operation on $A$, in a way which commutes with substitution.
\end{proposition}

\begin{proof}
  We assume that $A$ has a pointwise filling operation $c_A(X)$.
  We define then, for $f:Y\rightarrow X$
  \begin{equation*}
  \E(c_A)(X,\gamma,\psi,u)(i)(f) = c_A(Y)( \gamma f^+,\psi,u f^+)(i)
  \end{equation*}
  We can then check for $f:Y\rightarrow X$ and $g:Z\rightarrow Y$
  \begin{equation*}
  c_{\E A}(X,\gamma,\psi,u)(i)f(g) =
  c_A(Z)( \gamma (f g)^+,\psi,u (f g)^+)(i) =
  c_{\E A}(Y, \gamma f^+,\psi, u f^+)(i)(g)
  \end{equation*}
  and hence $c_{\E A}$ is natural in $X$.

  We can also define $\unitF$ in $\E \Unit$ by $\unitF(f) = \unit$
  and $\pairF{u}{v}:\E(\sum_A B)\rho$
  by $\pairF{u}{v}(f) = (u(f),v(f))$ for $u$ in $(\E A)\rho$ and
  $v$ in $(\EE B)(\rho,u)$, and check that all conditions for a lex operations are satisfied.

  Any universe $\UU$ reflects the operations $\E$ and $\EE$ since the Grothendieck universe $U$ used to construct $\UU$ was assumed to contain $\CC$.
\end{proof}

\begin{proposition}\label{Econtr}
  If $A$ is pointwise contractible then $\E A$ is contractible.
\end{proposition}
\begin{proof}
  We assume that $A$ has a pointwise extension operation $e_A(X)$.
  We define then, for $f:Y\rightarrow X$
  \begin{equation*}
  e_{\E A}(X,\rho,\psi,u)(f) = e_A(Y)(\rho f,\psi,uf)
  \end{equation*}
  We can then check for $f:Y\rightarrow X$ and $g:Z\rightarrow Y$
  \begin{equation*}
  e_{\E A}(X,\rho,\psi,u)f(g) =
  e_A(Z)(\rho f g,\psi,u f g) =
  e_{\E A}(Y,\rho f,\psi,u f)(g)
  \end{equation*}
  and hence $e_{\E A}$ is an extension operation for $\E A$ natural in $X$.
\end{proof}


In general, $\E$ may not be a descent data operation, since $\E A$ does not need to be $\E$\hyp{}modal.
The next \lcnamecref{sec:homotopy-descent-data} will use the lex operation $\E$ to define a descent data operation.

\subsection{Homotopy descent data}
\label{sec:homotopy-descent-data}

In this \lcnamecref{sec:homotopy-descent-data}, unless explicitly stated,
we work in the internal language of the presheaf model over $\CC\times\BB$.
Starting from the lex operation $\E$, we define a new lex operation
$\D$.
As before, we first define $\D$ on presheaves, and then show that it extends
to a lex operation on presheaves with a filling operation.
On presheaves with a filling operation,
$\D$ will be  a descent data operation.

We let $\PP_n$ be the subpresheaf
of $\II^{n+1}$ of elements $(i_0,i_1,\dots,i_n)$ satisfying $i_0=1\vee\dots\vee i_n=1$.

\medskip

Let $s_k:\II^{n+1}\rightarrow \II^n$ be the map which omits the $k$th component,
for $k = 0,\dots, n$. Note that $s_k\vec{i}$ is in $\PP_{n-1}$
if $\vec{i}$ is in $\PP_n$ and $i_k=0$.


\begin{definition}
  An element of $\D A$ is given by a family 
  $u(\vec{i})$ in $\E^{n+1}A$ defined on $\PP_n$
  and satisfying the {\em compatibility conditions}\footnote{It is suggestive to think of the elements
    of $\D A$ as {\em choice sequences} \cite{TroelstraD88b} extended in a {\em spatial} rather than
    {\em temporal} dimension.}
$u(\vec{i}) = \E^k(\alpha) u(s_k\vec{i})$ on $i_k=0$.
\end{definition}

For instance we have
\begin{align*}
u(0,i_1,i_2) &= \alpha u(i_1,i_2) &
u(i_0,0,i_2) &= \E(\alpha) u(i_0,i_2) &
u(i_0,i_1,0) &= \E^2(\alpha) u(i_0,i_1)
\end{align*}

We have an element $u(\vec{1})$ in each $\E^{n+1}A$.
We have a path $u(1,i)$ between $\alpha\,u(1)$ and $u(1,1)$
and a path $u(i,1)$ between $\E(\alpha)\,u(1)$ and $u(1,1)$ in $\E^2A$.
But, in general, we need further higher coherence conditions.

\medskip

We define $\eta_A : A\rightarrow \D A$ by $(\eta_A\, a)(i_0,i_1,\dots,i_n) = \alpha^{n+1} a$.

If $A$ is a family of types over $\Gamma$ we define $\D A$ family of types
over $\Gamma$ by $(\D A)\rho = \D(A\rho)$.

\begin{proposition}\label{Dfill}
  If $A$ is a family of types with a pointwise filling operation,
  then $\D A$ has a filling operation.
\end{proposition}

\begin{proof}
  We use that each $\E^{n+1}A$ has a (uniform) filling operation by \cref{Efill}
  hence is a family of types in the model over $\CC\times\BB$.
  We assume given $\gamma$ in $\Gamma^{\II}$ and $\psi$ in $\cofunivConst$ and a partial
  element $u_j$ in $(\D A)\gamma(j)$ defined over $\psi\vee j=0$. We explain
  how to define a total extension $v_j$ in $(\D A)\gamma(j)$.
  For this we define $v_j(\vec{i})$ in $\E^{n+1}A$  by induction on $n$.
  Since $\E^{n+1}A$ has a filling operation, we apply this filling operation
  to the partial element equal to $u_j(\vec{i})$ on $\psi\vee j=0$
  and equal to $\E^k(\alpha)\, {v_j}(s_k(\vec{i}))$ if $i_k = 0$.
\end{proof}

\begin{corollary}\label{Dlex}
  $\D$ defines a lex operation.
\qed
\end{corollary}

A similar argument as the one for \cref{Dfill}
using \cref{Econtr} instead proves the following.

\begin{proposition}\label{Dcontr} \leavevmode
\begin{parts}
\item
If $A$ is a pointwise contractible family of types over $\Gamma$, then $\D A$ is contractible.
\item
If $B$ is a pointwise contractible family of types over a family of types $A$ over $\Gamma$, then $\DD B$ is contractible over $\D A$.
\qed
\end{parts}
\end{proposition}

\begin{corollary}\label{Dequiv}
Let $\sigma : A \to B$ be map between fibrant families of types over $\Gamma$.
If $\sigma$ is pointwise an equivalence, then $\D\sigma$ is an equivalence.
\end{corollary}

\begin{proof}
  The fiber $\hFiber(\sigma)$ defines a pointwise contractible family
  of types over $B$. Hence $\DD \hFiber(\sigma)$ is contractible over $\D B$.
  Since $\D$ is a lex operation,  $\hFiber(\D\sigma)$ is contractible over $\D B$
  and $\D\sigma$ is an equivalence.
\end{proof}

\begin{proposition}\label{key1}
Let $A$ be a fibrant family of types over $\Gamma$.
Then $\eta_A$ is pointwise an equivalence and $\D\eta_A$ is an equivalence.
\end{proposition}

\begin{proof}
  For this \lcnamecref{key1}, we work in the presheaf model over $\BB$.
  If $\vec{f}$ is a composable chain of arrows we write $\mpF{\vec{f}}$ for its composition.

Let $A$ be a type over $\Gamma$. For $\rho$ in $\Gamma(X)$, an element $u$
  of $(\D A)\rho$ is a family of elements $u(\vec{i})(\vec{f})$ in $A\rho \mpF{\vec{f}}$
  satisfying   the compatibility conditions.
  For $a$ in $A\rho$ the element $\eta_A\, a$ is the family
  of element
  \begin{equation*}
  (\eta_A\, a)(\vec{i})(\vec{f}) = a\mpF{\vec{f}}
  \end{equation*}
  We define an inverse $G : DA(X) \to A(X)$ of $\eta_A(X)$ by taking $Gu$ to be the element
  $u(1)(\id_X)$. We then have $G(\eta_A\, a) = a$. The element
  $\eta_A\, (G\,u)$ satisfies
  \begin{equation*}
  (\eta_A\, (G\,u))(\vec{i})(\vec{f}) = (G\,u)\mpF{\vec{f}} = u(1)(\id)\mpF{\vec{f}}
  = u(1,\vec{0})(\id,\vec{f})
  \end{equation*}
  Define the element $\tilde{u}$ in $(\D A)\rho$ by
  $\tilde{u}(\vec{i})(\vec{f}) = u(1,\vec{i})(\id,\vec{f})$.
  We can define a homotopy
  \begin{equation*}
  u_k(\vec{i})(\vec{f}) = u(1,k\wedge \vec{i})(\id,\vec{f})
  \end{equation*}
  between $\eta_A\,(G\,u)$ and $\tilde{u}$ and we can define a homotopy
  \begin{equation*}
  v_k(\vec{i})(\vec{f}) = u(k,\vec{i})(\id,\vec{f})
  \end{equation*}
  between $u$ and $\tilde{u}$.%
  \footnote{At this point that we use that the object {\bf I} in $\BB$ has lattice
  operations but one could however instead define a homotopy in a more complex way by induction on
  the dimension for Cartesian cubes. The same remark applies for the proof of the next
  \lcnamecref{key2}.}
  By composition, there is a path between $u$ and $\eta_A\, (G\,u)$
  and $G$ is an inverse of $\eta_A(X)$.%

  This shows that $\eta_A$ is pointwise an equivalence.
  Then $D \eta_A$ is an equivalence by \cref{Dequiv}.
\end{proof}

One way to understand the definition of $\D$ from $\E$ is the following.
Being a pointed endofunctor, $\E$ defines a cosemisimplicial
diagram starting from $\E A$, and $\D A$ is a strict way to realize
the homotopy limit of this diagram using a $\PP$\hyp{}weighted limit.
We can think of $\PP$ as a cofibrant resolution of the constant diagram on $1$.
A remark is that $\E$, and hence each $\E^l$, preserves the $\PP$\hyp{}weighted limit defining $\D$.
In particular, an element of $\E^l(\D A)$ is determined
by a family 
$u(\vec{i})$ in $\E^{l+n+1}A$ satisfying
$u(\vec{i}) = \E^{l+k}(\alpha)\, u(s_k\vec{i})$ on $i_k = 0$.

\begin{proposition}\label{key2}
  Let $A$ be a fibrant family of types over $\Gamma$.
  We can build a path between $\eta_{\D A}$ and $\D\eta_A$.
\end{proposition}

\begin{proof}
  An element of $(\D^2A)\rho$ is given by a family $v(\vec{i})(\vec{j})$
  in $\E^{n+m+2}A$
  satisfying the conditions
\begin{enumerate}
\item $v(\vec{i})(\vec{j}) = \E^k(\alpha)\, v(s_k\vec{i})(\vec{j})$ on $i_k=0$
\item $v(\vec{i})(\vec{j}) = \E^{n+1+l}(\alpha)\, v(\vec{i})(s_l\vec{j})$ on $j_l=0$
\end{enumerate}
Given $u$ in $(\D A)\rho$ we define an element $\tilde{u}$ in $(\D^2A)\rho$
by $\tilde{u}(\vec{i})(\vec{j}) = u(\vec{i},\vec{j})$.

We compute, for $u$ in $(\D A)\rho$
\begin{equation*}
(\eta_{\D A}\, u)(\vec{i})(\vec{j}) = \alpha^{n+1}\, u(\vec{j}) = {u}(\vec{0},\vec{j})
\end{equation*}
and we have a homotopy connecting this map to $\tilde{u}$ by defining
\begin{equation*}
v_k(\vec{i})(\vec{j}) = u(\vec{i}\wedge k,\vec{j}).
\end{equation*}
We also have
\begin{equation*}
((\D\eta_A)\, u)(\vec{i})(\vec{j}) =
\E^{n+1}(\alpha^{m+1})\,u(\vec{i}) = {u}(\vec{i},\vec{0})
\end{equation*}
and we have a homotopy connecting this map to $\tilde{u}$ by defining
\begin{equation*}
w_k(\vec{i})(\vec{j}) = u(\vec{i},k\wedge \vec{j})
\end{equation*}
By composition, we have a path between $\D\eta_A$ and $\eta_{\D A}$.
\end{proof}

\begin{corollary}
$\D$ defines a descent data operation.
\end{corollary}

\begin{proof}
  By \cref{key1,key2}.
\end{proof}

Note that a direct consequence of \cref{Dequiv} is the following
strictification result.

\begin{theorem}\label{strict1}
  Let $A$ and $B$ be fibrant families of types over $\Gamma$ that are $\D$\hyp{}modal.
  Then any pointwise equivalence $\sigma : A \to B$ is an equivalence.
\qed
\end{theorem}

Let us note the following consequence of \cref{key1}.

\begin{corollary} \label{all-modal-via-pointwise-contractible}
The following conditions are equivalent:
\begin{conditions}
\item \label{all-modal-via-pointwise-contractible:modal}
all fibrant families of types are $D$-modal,
\item \label{all-modal-via-pointwise-contractible:euqivalence}
all pointwise equivalences between fibrant families of types are equivalences,
\item \label{all-modal-via-pointwise-contractible:contractible}
all fibrant families of types that are pointwise contractible are contractible.
\end{conditions}
\end{corollary}

\begin{proof}
The direction from \cref{all-modal-via-pointwise-contractible:modal} to \cref{all-modal-via-pointwise-contractible:euqivalence} is \cref{strict1}.
In the reverse direction, given a fibrant family of types $A$, recall that $\eta_A$ is a pointwise equivalence by \cref{key1}.
Then $\eta_A$ is an equivalence and hence $D$-modal.
\Cref{all-modal-via-pointwise-contractible:contractible} is a special case of \cref{all-modal-via-pointwise-contractible:euqivalence}.
The reverse direction holds since a (pointwise) equivalence can be described as a map with (pointwise) contractible fibers.
\end{proof}

The way from which we get $\D$ from $\E$ can also be applied to the
lex operation $\E A = A^R$, where $R$ is an arbitrary type.
This amounts to give a map which is {\em coherently constant} as defined by Kraus~\cite{Kraus15}
and so a map $\norm{R}\rightarrow A$ from the propositional truncation
of $R$ to $A$~\cite{Kraus15}.

Our development actually provides a way to recover
this result in the cubical setting.
Indeed, an element of $\D A$ is a sequence of elements $u(\vec{i})(\vec{x})$
in $A$ for $\vec{i}$ in $\PP_n$ and $\vec{x}$ in $R^{n+1}$ with
$u(\vec{i})(\vec{x}) = u(s_k\vec{i})(s_k\vec{x})$ on $i_k = 0$. Given an element
$x$ in $R$, we can build a left inverse $p_A$ of $\eta_A:A\rightarrow \D A$ by taking
$p_A u = u(1)(x)$. Hence we get an element of $R\rightarrow \isEquiv(\eta_A)$, and so of
$\norm{R}\rightarrow\isEquiv(\eta_A)$
which provides a factorization of a coherently constant map $R\rightarrow A$ through
$R\rightarrow\norm{R}$.

\subsection{Case of a monoid} \label{subsec:monoid}

 We consider the special case where the base category is a monoid $M$.
 If $\vec{x}$ is a sequence $(x_0,\dots,x_n)$,
 we write $t_k\vec{x}$ for the sequence where we omit $x_k$ and replace
 $x_{k+1}$ by $x_kx_{k+1}$ for $k<n$
 and $t_n\vec{x}$ is the sequence where we omit $x_n$. A type in the presheaf
 model is a type $A$ with an $M$\hyp{}action, and an element of $\D A$ is then
 a family of elements $u(\vec{i})(\vec{x})$ in $A$ with $\vec{i}$ in $\PP_n$
 and $\vec{x}$ in $M^{n+1}$ satisfying the compatibility conditions
 \begin{enumerate}
   \item $u(\vec{i})(\vec{x}) = u(s_k\vec{i})(t_k\vec{x})$ on $i_k = 0$ for $k<n$
     and
   \item $u(\vec{i})(\vec{x}) = u(s_n\vec{i})(t_n\vec{x})x_n$ on $i_n = 0$
     \end{enumerate}
 We define the $M$\hyp{}action on $\D A$ by
 $ux(\vec{i})(x_0,\dots,x_n) = u(\vec{i})(xx_0,\dots,x_n)$.

As a special case, let $M$ be the walking idempotent.
Let $e^2 = e$ be the non-trivial idempotent element of $M$.
Here is an example of a non\hyp{}modal type which is pointwise contractible, but not contractible.
Let $\Gamma$ be the set with elements $\rho_1,\rho_2$ and $\rho$ with
$\rho_1e = \rho_2 e = \rho$. We let $A$ be the following type.
We let $A\rho_1$ be the point $a_1$ and $A\rho_2$ be the point $a_2$
and $A\rho$ be the groupoid with two isomorphic objects $u_1,u_2$
with $a_ie = u_i$ for $i=1,2$. The type $A$ is then pointwise contractible
but it has no global point.%
\footnote{If $a$ is such a point, we should
have $a\rho_i = a_i$ and then $(a\rho_i) e = u_i$ and
$a(\rho_1e) = a(\rho_2e) = a\rho$ which is not possible since $u_1,u_2$
are distinct.}

\subsection{Generalization to a Grothendieck topology} \label{subsec:gen-grothendieck-top}

A Grothendieck topology $\JJ$ on the category $\CC$ defines
a set $\Cov(X,I) = \JJ(X)$ and we have a family
$\E_S$ indexed by $S:\Cov$ defined as follows. Let $\rho$ be in $\Gamma(X)$,
and $S$ is in $\Gamma\rightarrow\Cov$, so that
$S\rho$ is in $\Cov(X) = \JJ(X)$, which is a set of sieves on $X$.

 An element
of $(\E_SA)\rho$ is now a family $u(f)$ in $A\rho f$
{\em with $f$ in $S\rho$}.
We define in this way a family of lex operations $\E_S$
and an associated family of descent data operations $\D_S$
indexed by $S:\Cov$.

Note that if $S_1\rho$ is a subset of $S_2\rho$ for all $\rho$,
then we have a canonical projection map $\D_{S_2}A\rightarrow \D_{S_1}A$ that coheres with the pointings.
If $A$ is $\D_{S_1}$\hyp{}modal a left inverse of $\eta^{S_1}_A$ composed with this projection
map is a left inverse of $\eta^{S_2}_A$. Hence a  $\D_{S_1}$\hyp{}modal type
is also  $\D_{S_2}$\hyp{}modal and we have $\D_{S_1}\leqslant \D_{S_2}$ for the preorder defined
in \cref{Suniv}.
Since $\JJ$ is a Grothendieck topology, the family $\D_S$ over $S : \Cov$ is filtered.
Thus, we can apply \cref{filter} to obtain a model of univalent type theory with higher inductive types.
This can be seen as constructively modelling higher sheaves over $\JJ$ in the cubical model over $\BB$.

The next \lcnamecref{Gcontr} will be used for building such a sheaf model where
countable choice does not hold. The proof is similar to the one of \cref{Dfill}.

\begin{proposition}\label{Gcontr}
  If $A$ in $\Type(\Gamma)$ and $S$ in $\Gamma\rightarrow\Cov$ and $A$ is $D_S$-modal
  and $\rho$ in $\Gamma(X)$
  and $A\rho f$ is (pointwise) contractible for each $f$ in $S\rho$
  then we can find a uniform extension operation $e_{A\rho}(f,\psi,u)$ in $A\rho f$
  for all $f:Y\rightarrow X$ and $u$ partial element in $A\rho f$ of extent $\psi$.
\qed
\end{proposition}

By uniform, we mean that we have
\begin{equation*}
e_{A\rho}(f,\psi,u)g = e_{A\rho}(fg,\psi,ug)
\end{equation*}
in $A\rho fg$ for any $g:Z\rightarrow Y$.

\subsection{A model with the negation of countable choice}


  Using in an essential way the notion of {\em homotopy} descent data, we build a model
  of univalent type theory with higher inductive types
  with a countable family of sets $\E_n$ such that each the homotopy propositional truncation $\norm{\E_n}$ is inhabited,
  but $\norm{\prod_{n:N} \E_n}$ is not globally inhabited.

We consider the following space, corresponding to the lattice generated by formal elements $X_n$
and $L_n$ with the relations $X_0 = 1$, $X_n = L_n\vee X_{n+1}$ and $L_{n+1} = L_n\wedge X_{n+1}$.
Using \cref{Gcontr} one can show the following result.

\begin{proposition}\label{counter}
  The type $\norm{L_0+X_n}$ is contractible for all $n$
  while $L_0$ is the homotopy propositional truncation of $\prod_{n:N} (L_0 + X_n)$.
\qed
\end{proposition}

\begin{corollary}
  There exists a model of univalent type theory with higher inductive types where countable choice does not
  hold.
\qed
\end{corollary}

As stressed in~\cite{SwanU19}, it is yet unknown how to build a model of univalent type
theory and higher inductive types satisfying countable choice in a constructive metatheory.
(Countable choice holds in a classical metatheory in the simplicial set model.)





\section{Variation with another notion of cofibration} \label{sec:variation}

We explain how to modify the definition of filling operation if we
work with the other notion of cofibration classified by $\cofunivLW$.
Recall that an element of $\cofunivLW(X)$
is no longer constant, but
is given by a family of elements  $\psi_f$ in $\cofuniv_{\BB}$ for $f:Y\rightarrow X$ and
satisfying $\psi_{f}\leqslant \psi_{f g}$ if $g:Z\rightarrow Y$.

All the main results above still hold for this new notion of cofibration,
suitably modified.
The notion of {\em filling operation}
for $A$ is given by an operation $c_A$ which takes as
argument $\gamma$ in $\Gamma(X)^{\II_{\BB}}$ and $\psi$ in $\cofuniv_\BB(X)$
and a family of elements $u(i)$ in $A\gamma(i)f$ on the extent
$\psi_f \vee i = 0$ such that $u_f(i)g = u_{fg}(i)$  for $g:Z\rightarrow Y$
on the extent
$\psi_f \vee i = 0$.
(There is a dual operation with $i=1$ instead.)
It produces an element $c_A(X,\gamma,\psi,u)(i)$ in $A\gamma(i)$
such that

\begin{enumerate}
  \item $c_A(X,\gamma,\psi,u)(i)f = u_f(i)$ on $\psi_f\vee i=0$,
  \item $c_A(X,\gamma,\psi,u)(i)f = c_A(Y,\gamma',\psi f,u')(i)$ with
    $\gamma'(i) = \gamma(i)f$ and $u'_g(i) = u_{fg}(i)$ on the extent
    $\psi_{fg}\vee i=0$ for $g:Z\rightarrow Y$.
\end{enumerate}

For instance, \cref{Efill} becomes the following result.

\begin{lemma}
  If $A$ has a pointwise filling operation $c_A(X)$ then $\E A$ has a filling operation.
\end{lemma}

\begin{proof}
  We take $\gamma$ in $\Gamma(X)^{\II_{\BB}}$ and $u_f(i)$ in $(\E A)\gamma(i)f$
  on the extent $\psi_f\vee i = 0$ and we define $v(i) = c_{\E A}(X,\gamma,\psi,u)(i)$
  in $(\E A)\gamma(i)$. For $f:Y\rightarrow X$, we take (filling at level $Y$)
  \begin{equation*}
  v(i)(f) = c_A(Y)(\gamma',\psi',u')
  \end{equation*}
  where $\gamma'(i) = \gamma(i)f$ and $\psi' = \psi_f$
  and $u'(i) = u_f(i)(\id_Y)$ in $A\gamma(i)f$ on the extent $\psi_f\vee i=0$.
\end{proof}

Let us give some examples.

\medskip

The first example is when $\CC$ is the poset $0 \leqslant 1$.
In this case, a global type $A$ is given by two spaces with a map $A(1)\rightarrow A(0)$.
An element of $\cofunivLW(0)$ is an element of $\cofuniv_{\BB}$ while an element of $\cofunivLW(1)$ is a pair $\psi_1,\psi_0$ of elements of $\cofuniv_{\BB}$ with $\psi_1\leqslant\psi_0$.
One can check that $A$ is fibrant exactly if $A(0)$ is fibrant and $A(1) \to A(0)$ is a fibration, and a similar characterization holds in the relative situation (for a type $A$ over $\Gamma$) and for trivial fibrations.
Using \cref{all-modal-via-pointwise-contractible:contractible} of \cref{all-modal-via-pointwise-contractible}, one sees that every type in the model is $\D$\hyp{}modal.
The model coincides with the Reedy presheaf model described in~\cite{Shulman15} over the direct category $\CC$ in the model of univalent type theory given by the base model.
More generally, this will be the case for an arbitrary direct category $\CC$ for which the inclusion of objects into morphisms given by identities is decidable.

\medskip

The second example is the walking retract $\CC$ generated by maps $f \colon 0 \to 1$ and $g \colon 1 \to 0$ satisfying $g f = \id_0$.
Note that $\CC$ is the idempotent splitting of the walking idempotent monoid $\M$ considered in \cref{subsec:monoid}.
This makes the cubical presheaf models (for both $\cofunivConst$ and $\cofunivLW$) over $\CC$ and $\M$ equivalent.
Level $0$ in the model over $\CC$ correspond to the fixpoints of the action of $e$ in the model over $\M$.
Taking $\cofunivLW$ as the cofibration classifier, the model of modal types gives a model for pointed families in a cubical model.
It is homotopically correct in the sense that the equivalences are levelwise.

\medskip

One might ask if types in the above model are already $\D$\hyp{}modal, similar to what happens for the poset $0 \leqslant 1$.
More generally, one might attempt to generalize from a direct category $\CC$ to a Reedy category $\CC$ that is elegant~\cite{BergnerR13}; the walking retract is an example of an elegant Reedy category, with coface map $f$ and codegeneracy map $g$.
Taking $\cofunivLW$ as the cofibration classifier, one might ask if the (trivial) fibrations are given by the (trivial) Reedy fibrations; as before, this would imply that every type in the model is $\D$\hyp{}modal.
An equivalent condition is that the levelwise cofibrations (classified by $\cofunivLW$) are also the Reedy cofibrations.
This holds true in classical situations where cofibrations and monomorphisms coincide and gives rise to the classical model~\cite{Shulman13} over an elegant Reedy category.

Unfortunately, this fails to hold in our constructive setting.
Ultimately, this is because the inclusions $A(X) \to A(Y)$ are not generally cofibrations for a global type $A$ and a codegeneracy map $Y \to X$ in $\CC$.
For the case of the walking retract, this is the inclusion $A(0) \to A(1)$.
In terms of a global type $A$ in the model over the walking monoid $\M$, it is the inclusion of fixpoints of the action of $e$ on $A$.
For a counterexample, let $S$ be a discrete space with non-decidable equality in one of the concrete cubical models listed in~\cref{subsec:examples}.
Take $A = S \times S$ with the action of $e$ given by swapping.

\section{Related and future work}

Shulman~\cite{Shulman19} shows that all $(\infty,1)$\hyp{}toposes have strict univalent universes, using a classical metatheory.
This work does not cover however (yet) higher inductive types and cumulativity of universes.
There are close connections between Shulman's work and ours, which we plan to explore in future work.
His work inspired some results about pointwise weak equivalences in~\cref{sec:homotopy-descent-data}, in particular~\cref{Dequiv}.

Once we have a presheaf model of univalence with homotopical features such as ours, it is now understood (see \eg~\cite{Sattler17,Boulier18}) how to define a Quillen model structure whose (trivial) fibrations coincide with the (contractible) types.
For the model of $\D$\hyp{}modal types, we expect that, similar to~\cite{Shulman19}, that the weak equivalences are the levelwise weak equivalences and the fibrations are a variation%
\footnote{We define a family of types to be injectively fibrant if it lifts against cofibrations that are levelwise trivial cofibrations.}
of the injective fibrations.
We leave this to future work.

Instead of parameterizing our construction over an external category $\CC$, we could start from a internal category $\CC$ in presheaves over $\BB$.
Note that the category of presheaves over an internal category in presheaves is still a presheaf category.
Compared to the construction of~\cite{Shulman19} (which instantiates at this level of generality), we seem to need less fibrancy assumptions on this internal category.
We leave this generalization to future work.



\section*{Acknowledgements}

Many thanks to Mathieu Anel, Steve Awodey, Mart\'{\i}n Escard\'{o}, Eric Finster,
Dan Licata, Emily Riehl, Mike Shulman, Bas
Spitters and Matthew Weaver for many discussions and remarks.

\appendix

\section{General results for lex modalities}\label{sec:general-results}

Some of our results hold for modalities in the sense of~\cite{RijkeSS17} that are not necessarily presented in a strict manner by a lex operation.
The main example is the case of accessible modalities, which are implemented using higher inductive types that rarely give rise to a lex operation.
The purpose of this \lcnamecref{sec:general-results} is to prove these more general statements.
We work in the homotopy type theory setting of~\cite{RijkeSS17}.
Universes are assumed univalent and closed under dependent sums, dependent products, identity types.
For statements involving accessible modalities, we also assume closure under higher inductive types.

In this \lcnamecref{sec:general-results}, we take terminology with potentially both strict and homotopical meaning to have the homotopical meaning by default.
This is opposed to the rest of the article, where we default to the strict meaning.
For example, equality refers to the identity type, and pullbacks refer to homotopy pullbacks (expressed using the identity type).

We write $\Modality(\UU)$ for the type of modalities on a universe $\UU$.
Recall from~\cite{RijkeSS17} that $M : \Modality(\UU)$ has an underlying subuniverse%
\footnote{%
By a subuniverse of $\UU$, we mean a subobject of $\UU$, \ie a predicate on $\UU$.
This is formally a map $\UU \to \Prop$ where $\Prop$ is the universe of (homotopy) propositions.
It is not to be confused with a subuniverse in the set\hyp{}theoretic sense in a model where universes are built out of sets.
We note that the size of the propositions in $\Prop$ here does not matter for us; one choice is propositions in $\UU$, but one could allow also a larger universe.
}
of $\UU$, the $M$\hyp{}modal types $\UU_M$.
Subuniverses of $\UU$ carry an evident poset structure.
Following~\cite[Subsection~3.2]{RijkeSS17}, we obtain a poset structure also on $\Modality(\UU)$.

\begin{definition} \label{modality-extension}
Let $\UU$ be a universe contained in a universe $\UU'$.
A modality $M'$ on $\UU'$ is an \emph{extension} of a modality $M$ on $\UU$ if every $M$\hyp{}modal type in $\UU$ is $M'$\hyp{}modal in $\UU'$ and for $X : \UU$, the canonical map $M'X \to MX$ is invertible.
\end{definition}

The above conditions mean that a $\UU$\hyp{}small type is $M$\hyp{}modal exactly if it is $M'$\hyp{}modal and $M$\hyp{}connected exactly if it is $M'$\hyp{}connected.
In terms of the stable factorization systems $(\mathcal{L}, \mathcal{R})$ and $(\mathcal{L}', \mathcal{R}')$ corresponding to $M$ and $M'$, this means that $\mathcal{L}$ and $\mathcal{R}$ are the restrictions of $\mathcal{L}'$ and $\mathcal{R}'$ to maps between $\UU$\hyp{}small types.
For this, recall~\cite[Subsection~1.2]{RijkeSS17} that the left and right classes of the stable factorization system corresponding to a modality are the connected and modal maps, which are defined by having connected and modal fibers, respectively.

We write $\Modality(\UU < \UU')$ for the type of pairs $(M, M')$ with $M$ a modality on $\UU$ and $M'$ an extension of $M$ to $\UU'$.
The poset structures on $\Modality(\UU)$ and $\Modality(\UU')$ extend to a poset structure on $\Modality(\UU < \UU')$.

The following statement makes precise that up to (essential) size issues, a modality is lex exactly if the universe of modal types is modal.
In particular, a ``size\hyp{}polymorphic'' modality (acting compatibly on all universes) whose action on maps preserves smallness of fibers is lex exactly if universes of modal types are modal.
This generalizes~\cref{univ} to modalities; the smallness condition on fibers mirrors the dependent action $\DD$ on $\UU$\hyp{}small types we require for a lex operation $\D$.
For $M : \Modality(\UU)$, we denote by $\UU_M$ the subuniverse of $\UU$ of $M$\hyp{}modal types.
\begin{proposition} \label{modality-lex-via-universe}
For $(M, M') : \Modality(\UU < \UU')$:
\begin{parts}
\item \label{modality-lex-via-universe:lex-to-universe}
if $M'$ is lex and preserves maps with $\UU$\hyp{}small fibers, then $\UU_M$ is $M'$\hyp{}modal;
\item \label{modality-lex-via-universe:universe-to-lex}
if $\UU_M$ is $M'$\hyp{}modal, then $M$ is lex and $M'$ preserves maps with $\UU$\hyp{}small fibers.
\end{parts}
\end{proposition}

\begin{proof}
For \cref{modality-lex-via-universe:lex-to-universe}, let $M'$ be lex and preserve maps with $\UU$\hyp{}small fibers.
To show that $\UU_M$ is $M'$\hyp{}modal, it suffices to construct a left inverse to $\eta^{M'}_{\UU_M}$ (\cite[Lemma~1.20]{RijkeSS17}).
By univalence of $\UU_M$, this means to find an extension
\[
\xymatrix{
  \textstyle\sum_{X:\UU_M} X
  \ar@{.>}[r]
  \ar[d]
  \pullback{dr}
&
  Z
  \ar@{.>}[d]^{\txt\scriptsize{$\UU$\hyp{}small and\\ $M$\hyp{}modal fibers}}
\\
  \UU_M
  \ar[r]^{\eta^{M'}_{\UU_M}}
&
  M' \UU_M
\rlap{.}}
\]
We use the naturality square of $\eta^{M'}$ at the left map.
The square is a pullback because $M'$ is lex.
The right map has $\UU$\hyp{}small fibers by assumption and has $M$\hyp{}modal fibers because it is $M'$\hyp{}modal as it goes between $M'$\hyp{}modal types.

For \cref{modality-lex-via-universe:universe-to-lex}, assume that $\UU_M$ is $M'$\hyp{}modal.
Then $\UU_M$ is right orthogonal against $M'$\hyp{}connected types, in particular $M$\hyp{}connected types.
This verifies condition~(xiii) of~\cite[Theorem~3.1]{RijkeSS17}, making $M$ lex.
It remains to show that $M'$ preserves maps with $\UU$\hyp{}small fibers.
Given such a map, we factor it using $M$ as an $M'$\hyp{}connected map followed by a map with fibers in $\UU_M$.
Since $M'$ sends $M'$\hyp{}connected maps to equivalences, it remains to show, given $Y : X \to \UU_M$, that $M'(\sum_X Y) \to M'X$ has $\UU$\hyp{}small fibers.
Since $\UU_M$ is $M'$\hyp{}modal, it is right orthogonal against $X \to M'X$.
Thus, $Y : X \to \UU_M$ extends uniquely to a map $Y' : M'X \to \UU_M$.
Looking at the classified maps, we obtain the following commuting diagram:
\[
\xymatrix{
  \textstyle\sum_X Y
  \ar[r]
  \ar[d]
  \pullback{dr}
&
  \textstyle\sum_{z:M'X} Y'(z)
  \ar[d]
\\
  X
  \ar[r]^{\eta^{M'}_X}
&
  M'X
\rlap{.}}
\]
Since the right map has $M$\hyp{}modal (hence also $M'$\hyp{}modal) fibers, it is $M'$\hyp{}modal.
The top map is a pullback of $\eta^{M'}_X$, hence $M'$\hyp{}connected.
Since
\[
\textstyle\sum_X Y \longrightarrow \textstyle\sum_{z:M'X} Y'(z) \longrightarrow M'X
\]
and
\[
\textstyle\sum_X Y \longrightarrow M'(\textstyle\sum_X Y) \longrightarrow M'X
\]
are ($M'$\hyp{}connected, $M'$\hyp{}modal)\hyp{}factorizations of the same map, they coincide.
This shows that the map $M'(\sum_X Y) \to M'X$ is equal to $\sum_{z:M'X} Y'(z) \to M'X$, hence has $\UU$\hyp{}small fibers.
\end{proof}

Recall from~\cite[Subsection~2.3]{RijkeSS17} that accessible modalities admit canonical extensions to larger universes.
If the accessible modality is lex, we observe that it satisfies the technical condition on smallness of fibers of \cref{modality-lex-via-universe}.
This means that \cref{modality-lex-via-universe:lex-to-universe} of that statement can also be regarded as a generalization of the direction from condition~(i) to condition~(iii) in~\cite[Theorem~3.11]{RijkeSS17}.

\begin{corollary}
Let $M$ be an accessible lex modality on a universe $\UU$.
Let $M'$ be its extension to a universe $\UU'$ containing $\UU$.
The $M'$ preserves maps with $\UU$\hyp{}small fibers.
\end{corollary}

\begin{proof}
This follows from \cref{modality-lex-via-universe:universe-to-lex} of \cref{modality-lex-via-universe} since $\UU_M$ is $M'$\hyp{}modal by~\cite[Theorem~3.11]{RijkeSS17}.
\end{proof}

Let $M : I \to \Modality(\UU)$ be a family of modalities.
We write
\begin{equation} \label{subuniverse-structural-meet}
\UU(M) = \textstyle\sum_{X : \UU} \textstyle\prod_{i : I} \,\text{$X$ is $M_i$\hyp{}modal}.
\end{equation}
for the meet of the subuniverses of modal types of $M_i$ over $i : I$.
We call a given meet $\bigwedge M$ of $M$ \emph{structural} if it is preserved under the forgetful functor to the poset of subuniverses.
This means that its subuniverse of modal types is $\UU(M)$.
By~\cite[Theorem 3.11, part (i)]{RijkeSS17}, $M$ has a structural meet exactly if $\UU_M$ admits a reflection in $\UU$.
In that case, $\bigwedge M$ is given by the reflection operation.

Given a family $(M, M') : I \to \Modality(\UU < \UU')$, we say that a given meet of $(M, M')$ is structural if it is sent to structural meets of $M$ and $M'$ by the forgetful functors.
Note that $(M, M')$ has a structural meet exactly if $M$ and $M'$ have structural meets $\bigwedge M$ and $\bigwedge M'$, respectively, and $\bigwedge M'$ is an extension of $\bigwedge M$ to $\UU'$.
This unfolds to the following conditions:
\begin{itemize}
\item
the subuniverse $\UU_M$ of $\UU$ admits a reflection $L$,
\item
the subuniverse $\UU'_{M'}$ of $\UU'$ admits a reflection $L'$,
\item
for $X : \UU$, the canonical map $L'X \to LX$ is invertible.
\end{itemize}

When considering diagrams in a poset, we will restrict our attention to shapes that are themselves posets.
Note that in any poset, the limit of a (poset\hyp{}indexed) diagram coincides with the meet over the object components of the diagram.
Nonetheless, it is useful to speak about limits of diagrams because this allows us to constrain the relations between the inputs objects.

A poset $I$ is \emph{filtered} if it is merely inhabited and for any two elements $x_0, x_1 : I$, there merely exists $y : I$ with $x_0, x_1 \leq y$.
It is \emph{cofiltered} if $I^\op$ is filtered.
The following statement generalizes \cref{filter} to modalities.

\begin{proposition} \label{cofiltered-limit-subuniverses}
Let $(M, M') : I \to \Modality(\UU < \UU')$ be a $\UU$\hyp{}small cofiltered diagram.
If $\UU_{M_i}$ is $M'_i$\hyp{}modal for all $i : I$, then $\UU(M) : \UU'$ belongs to $\UU'(M')$.
\end{proposition}

\begin{proof}
Given $i : I$, we have to show that $\UU(M)$ is $M'_i$\hyp{}modal.
Because $I$ is cofiltered, we have
\[
\UU(M) = \UU\parens[\big]{(M_j)_{j \leq i}}
,\]
so it suffices to show that $\UU\parens[\big]{(M_j)_{j \leq i}}$ is $M'_i$\hyp{}modal.
By assumption, $\UU_{M_j}$ is $M'_j$\hyp{}modal, hence $M'_i$\hyp{}modal for $j \leq i$.
We now use that a type $X$ over $\UU_{M_i}$ ($M'_i$\hyp{}modal) is $M'_i$\hyp{}modal exactly if the map $X \to \UU_{M_i}$ is $M'_i$\hyp{}modal.
Given that $\UU_{M_j} \to \UU_{M_i}$ is $M'_i$\hyp{}modal for $j \leq i$, it suffices to show that $\UU\parens[\big]{(M_j)_{j \leq i}} \to \UU_{M_i}$ is $M'_i$\hyp{}modal.
Observe that the fibers of the latter embedding are products of the fibers of the former embeddings.
So the claim holds since modal types are closed under product (\cite[Lemma~1.26]{RijkeSS17}).
\end{proof}

\begin{corollary} \label{cofiltered-limit-universes-of-modals}
Let $(M, M') : I \to \Modality(\UU < \UU')$ be a $\UU$\hyp{}small cofiltered diagram with a structural meet $(\bigwedge M, \bigwedge M')$.
If $\UU_{M_i}$ is $M'_i$\hyp{}modal for all $i : I$, then $\UU_{\bigwedge M}$ is $\bigwedge M'$\hyp{}modal.
\end{corollary}

\begin{proof}
This is a direct consequence of \cref{cofiltered-limit-subuniverses} and the definition of structural meet.
\end{proof}

The following statement says that, up to the same size issues of \cref{modality-lex-via-universe}, lex modalities are closed under structural cofiltered limits of modalities.
In particular, structural cofiltered limits of ``size\hyp{}polymorphic'' modalities whose actions on maps preserve smallness of fibers preserve left exactness.

\begin{corollary} \label{cofiltered-limit-lex}
In the situation of \cref{cofiltered-limit-universes-of-modals}, if $M'_i$ is lex for $i : I$ and preserves maps with $\UU$\hyp{}small fibers, then $\bigwedge M$ is lex and $\bigwedge M'$ preserves maps with $\UU$\hyp{}small fibers.
\end{corollary}

\begin{proof}
This is the combination of \cref{modality-lex-via-universe} and \cref{cofiltered-limit-universes-of-modals}.
\end{proof}

Finally, we specialize to the important case of accessible modalities.

\begin{corollary} \label{cofiltered-limit-lex-accessible}
Let $M : I \to \Modality(\UU)$ be a $\UU$\hyp{}small cofiltered diagram.
If $M_i$ is lex and accessible for all $i : I$, then the meet $\bigwedge M$ exists and also has these properties.
\end{corollary}

\begin{proof}
Let $\UU'$ be a universe containing $\UU$.
Let $(M, M') : I \to \Modality(\UU < \UU')$ be the extension of $M$ given by~\cite[Theorem~3.36]{RijkeSS17}.
By~\cite[Theorem~3.29]{RijkeSS17}, the meet of $(M, M')$ exists, is structural, and $\bigwedge M$ is again accessible.
By~\cite[Theorem~3.11]{RijkeSS17}, $\UU_{M_i}$ is $M'_i$\hyp{}modal for $i : I$.
Applying \cref{cofiltered-limit-universes-of-modals}, $\UU_{\bigwedge M}$ is $\bigwedge M'$\hyp{}modal.
By~\cite[Theorem~3.29]{RijkeSS17}, this makes $\bigwedge M$ is lex.
\end{proof}

\end{document}